\documentclass[preprint,12pt]{elsarticle}

\usepackage{latexsym}
\usepackage{amsmath}
\usepackage{amssymb}
\usepackage{amsfonts}
\usepackage{mathrsfs}

\newtheorem{thm}{Theorem}

\newtheorem{lem}[thm]{Lemma}
\newtheorem{prop}[thm]{Proposition}
\newtheorem{defn}[thm]{Definition}
\newenvironment{proof}[1][Proof]{\textbf{#1.}\ }{\ $\Box$}

\journal{Journal of Differential Equations}
\begin{document}
\begin{frontmatter}

\title{On a classification of polynomial differential operators }

\author{Jinzhi Lei }

\address{Zhou Pei-Yuan Center for Applied Mathematics,  Tsinghua University, Beijing, 100084, P.R.China}

\begin{abstract}
This paper gives a classification of first order polynomial differential operators of form $\mathscr{X} = X_1(x_1,x_2)\delta_1 + X_2(x_1,x_2)\delta_2$, $(\delta_i = \partial/\partial x_i)$.  The classification is given through the order of an operator that is defined in this paper.  Let $X=\mathscr{X}y$ to be the differential polynomial associated with $\mathscr{X}$,  the order of $\mathscr{X}$, $\mathrm{ord}(\mathscr{X})$, is defined as the order of a differential ideal $\Lambda$ of differential polynomials that is a nontrivial expansion of the ideal $\{X\}$ and with the lowest order.  In this paper, we prove that there are only four possible values for the  order of a differential operator, $0$, $1$, $2$, $3$, or $\infty$.  Furthermore, when the order is finite,  the expansion $\Lambda$ is generated by $X$ and a differential polynomial $A$, which can be obtained through a rational solution of a partial differential equation that is given explicitly in this paper. When the order is infinite, the expansion $\Lambda$ is just the unit ideal. In additional, if, and only if, the order of $\mathscr{X}$ is $0$, $1$, or $2$, the polynomial differential equation associating with $\mathscr{X}$ has Liouvillian first integrals. Examples for each class of differential operators are given at the end of this paper. 
\end{abstract}


\end{frontmatter}

\section{Introduction}
\label{sec:0}

\subsection{Background}
This paper studies the polynomial differential operator
\begin{equation}
\label{eq:1}
\mathscr{X}  = X_1(x_1,x_2)\delta_1 + X_2(x_1,x_2)\delta_2,
\end{equation}
where $\delta_i  = \partial/\partial x_i\ (i=1,2)$, and $X_1(x_1,x_2), X_2(x_1,x_2)$ are polynomials of $x_1$ and $x_2$.  We further assume that $X_1\not\equiv 0$ without loss of generality. We will give a classification for all operators of form \eqref{eq:1}, according to which the solution of the first order partial differential equation
\begin{equation}
\label{eq:2}
\mathscr{X}\omega = 0
\end{equation}
is discussed. 

The operator \eqref{eq:1} closely relates to the following polynomial differential equation
\begin{equation}
\label{eq:3}
\dfrac{d x_1}{dt} = X_1(x_1,x_2),\quad \dfrac{d x_2}{dt} = X_2(x_1,x_2), 
\end{equation}
and non constant solutions of \eqref{eq:2} give first integrals of \eqref{eq:3}.  Therefore, our results also yield a classification of the polynomial systems \eqref{eq:3}. 

The current study was motivated by investigating integrating methods of a polynomial differential equation of form \eqref{eq:3}. We first look at a simple situation. 
If the equation \eqref{eq:3} has an integrating factor $\mu$ which is a rational function of $x_1$ and $x_2$, a first integral $\omega$ of \eqref{eq:3} can be obtained by an integral of a rational function, and further, we have
\begin{equation}
\label{eq:if2}
\delta_1 \omega - a = 0,
\end{equation}
where $a = \mu X_1$ is a rational function. Therefore, there is a non constant function $\omega$ that satisfies both equations \eqref{eq:2} and \eqref{eq:if2}. In this case, the differential operator $\mathscr{D}_A$ defined as   
\begin{equation}
\mathscr{D}_A\omega = \delta_1 \omega - a
\end{equation}
is compatible with $\mathscr{X}$. In other words, if we define two differential polynomials 
$$X = \mathscr{X}y,\quad A = \mathscr{D}_Ay,$$
they can generate a differential ideal $\{X,A\}$ which is a nontrivial expansion of the ideal $\{X\}$ (refer detail definitions below). This simple situation suggests that to integrate the equation \eqref{eq:3} for first integrals,  we need  to find a differential polynomial $A$ such that $\{X,A\}$ is a nontrivial expansion of the ideal $\{X\}$. The differential polynomial $A$, if exist, is not unique. Nevertheless, we will show that the lowest order among these differential polynomials is uniquely determined by the original differential operator $\mathscr{X}$ (called the order of $\mathscr{X}$, to be detailed below), and therefore provides a classification. 

The classification presented in this study is obtained from the order of the operator $\mathscr{X}$. This order is essential for understanding integrating methods the polynomial differential equation \eqref{eq:3} in different classes, and also the classification of un-integrable systems. Furthermore, for a given equation \eqref{eq:3}, the above differential polynomial $A$ in defining the nontrivial expansion $\{X,A\}$ provides additional informations for the first integral, which are important for further investigations of the structure of integrating curves (or foliations) of the equation.  Applications based on the classification given here is interested in future studies.

\subsection{Preliminary definitions}
Before stating the main results, we give some preliminary concepts from differential algebra. For detail discussions, refer\cite{Ka:76} and \cite{Ritt:50}.

Let $K$ to be the field of all rational functions of $(x_1, x_2)$ with complex number coefficients,  and $\delta_1, \delta_2$ are two \textit{derivations} of $K$.  Then $K$ together with the two derivations form a \textit{differential field}, with $\mathbb{C}$ as the constant field. For a \textit{differential indeterminate} $y$, there is a usual way to add $y$ to the differential field $K$, by adding an infinite sequence of symbols
\begin{equation}
\label{eq:4}
y, \delta_1y, \delta_2y,\delta_1\delta_2y,\cdots, \delta_1^{i_1}\delta_2^{i_2}y,\cdots
\end{equation}
to $K$ \cite{Ka:76}.  This procedure results in a differential ring, denoted as $K\{y\}$.  Each  element in $K\{y\}$ is a polynomial of finite numbers of the symbols in \eqref{eq:4}, and therefore is a \textit{differential polynomial} in $y$ with coefficients in $K$. 

We say an algebra ideal $\Lambda$ in $K\{y\}$ to be a \textit{differential ideal} if $a\in \Lambda$ implies $\delta_i a \in \Lambda\ (i=1,2)$.   Let $\Sigma$ be any aggregate of differential polynomials. The intersection of all differential ideals containing  $\Sigma$ is called the \textit{differential ideal generated by $\Sigma$}, and is denoted by $\{\Sigma\}$. A differential polynomial $A$ is in $\{\Sigma\}$ if, and only if, $A$ is a linear combination of differential polynomials in $\Sigma$ and of derivatives, of various orders, of such differential polynomials.

\begin{defn}
Let
$$w_1= \delta_1^{i_1}\delta_2^{i_2}y,\quad w_2 = \delta_1^{j_1}\delta_2^{j_2}y,$$
be two derivatives of $y$, $w_2$ is \textbf{higher} than $w_1$ if either $j_1>i_1$, or $j_1 = i_1$ and $j_2  > i_2$. The indeterminate $y$ is always higher than any element in $K$.
\end{defn}

\begin{defn}
Let $A$ be a differential polynomial, if $A$ contains $y$ (or its derivatives) effectively,  by the \textbf{leader} of $A$, we mean the highest of those derivatives of $y$ which are involved in $A$.  
\end{defn}

\begin{defn}
Let $A_1, A_2$ be two differential polynomials, we say $A_2$ to be of higher \textbf{rank} than $A_1$, if either
\begin{enumerate}
\item[(1)] $A_2$ has higher leader than $A_1$; or
\item[(2)] $A_1$ and $A_2$ have the same leader, and the degree of $A_2$ in the leader exceeds that of $A_1$.
\end{enumerate}
A differential polynomial which effectively involves the intermediate $y$ will be of higher rank than one which does not.
Two differential polynomials of which no difference in the rank as created above will be said to be of the same rank.
\end{defn}

Following fact is basic \cite[pp.~3]{Ritt:50}:

\begin{prop}
\label{prop:1}
Every aggregate of differential polynomials contains a differential polynomial which is not higher than any other differential polynomials in the aggregate.
\end{prop}

For the operator $\mathscr{X}$ given by \eqref{eq:1}, we have 
\begin{equation}
\label{eq:5}
X= \mathscr{X}y = X_1 \delta_1y+X_2\delta_2y\in K\{y\}.
\end{equation}
Let $\{X\}$ denote the differential ideal in $K\{y\}$ that is generated by $X$.  A differential ideal $\Lambda$ in $K\{y\}$ that contains $\{X\}$ as a proper subset will be called an \textit{expansion of $\{X\}$}, or an \textit{expansion of $\mathscr{X}$}.  In this paper, we will show that expansions of $\mathscr{X}$ with the lowest order (to be defined below) will be essential to provide the classification of $\mathscr{X}$.

Let $\Lambda$ to be an expansion of $\{X\}$, Proposition \ref{prop:1} yields that there is a differential polynomial $A\in \Lambda$ that has the lowest rank. Therefore, the leader of $A$ is lower than the leader of $X$, $\delta_1y$. Thus, either the leader of $A$ has form $\delta_2^r y\ (r\geq 0)$,  or $A$ does not involve the intermediate $y$, i.e., $A\in K$.  In the former situation, $r$ will be called the \textbf{\textit{order}} of $\Lambda$,  denoted by $\mathrm{ord}(\Lambda)$.  The latter situation will be called to have order of infinity, i.e., $\mathrm{ord}(\Lambda)=\infty$. 

For a differential polynomial $A\in K\{y\}$, we associate with $A$ a \textit{differential operator} $\mathscr{D}_A$ on analytic functions $\mathcal{A}(\Omega)$, where $\Omega$ is an open subset of $\mathbb{C}^2$, such that
\begin{equation}
\label{eq:oa}
\mathscr{D}_Au = A|_{y=u},\quad \forall u\in \mathcal{A}(\Omega).
\end{equation}
By $S(\mathscr{D}_A)$, we denote the singularity set of $\mathscr{D}_A$, which contains all singularity points in the coefficients of the differential polynomial $A$. Because the coefficients of $A$ are rational functions, the singularity set $\mathscr{D}_A$ is a closed subset in $\mathbb{C}^2$.  Thus, for any $u\in \mathcal{A}(\Omega)$, $\mathscr{D}_A u$ is well defined in the open subset $\Omega\backslash S(\mathscr{D}_A)$.

We will called an expansion of $\mathscr{X}$, $\Lambda$, to be \textit{nontrivial} if there exists an open subset $\Omega\subset \mathbb{C}^2$ and a non constant function $\omega\in \mathcal{A}(\Omega)$, such that $\mathscr{D}_Au = 0$ in $\Omega\backslash S(\mathscr{D}_A)$ for all $A\in \Lambda$. Otherwise, the expansion is called \textit{trivial}.
Examples of trivial expansion include $\{X, p(y)\}$ with $p(y)$ a proper polynomial of $y$ with constant coefficients (not a differential polynomial).

For a nontrivial expansion $\Lambda$,  a differential polynomial with the lowest rank can only take one of the following forms:
\begin{itemize}
\item a polynomial of $y$, with at least one coefficient that is non constant ($\mathrm{ord}(\Lambda) = 0$); or
\item a differential polynomial of $y$ effectively involves derivatives ($1\leq \mathrm{ord}(\Lambda) < \infty$); or 
\item an element in $K$, and therefore $\Lambda = K\{y\}$ ($\mathrm{ord}(\Lambda) = \infty$).
\end{itemize}

\subsection{Main results}

In this paper, we are interested at nontrivial expansions of $\mathscr{X}$ with the lowest order, called \textit{essential expansions of $\mathscr{X}$}.  For a given differential operator $\mathscr{X}$, essential expansions of $\mathscr{X}$ may not unique, but all essential expansions have the same order, which we call the \textit{\textbf{order} of $\mathscr{X}$}, and is denoted as $\mathrm{ord}(\mathscr{X})$.  We will show that $\mathrm{ord}(\mathscr{X})$ provides a classification of polynomial differential operators.

\begin{thm}
\label{th:1} Let the polynomial differential operator $\mathscr{X}$ given by \eqref{eq:1}, with coefficients $X_1,X_2\in K$,  then either
$0\leq \mathrm{ord}(\mathscr{X})\leq 3$, or $\mathrm{ord}(\mathscr{X}) = \infty$.
Furthermore, when $0\leq \mathrm{ord}(\mathscr{X})\leq 3$, we can always select an essential expansion $\Lambda$ of $\mathscr{X}$, such that $\Lambda = \{X, A\}$, with $A\in K\{y\}$ given below
\begin{enumerate}
\item[(1)] if $\mathrm{ord}(\mathscr{X}) = 0$, then
\begin{equation}
\label{eq:7}
A =  y - a,\quad (a\in K\backslash\mathbb{R});
\end{equation}
\item[(2)] if $\mathrm{ord}(\mathscr{X}) = 1$, then
\begin{equation}
\label{eq:8}
A=(\delta_2y)^n - a,\quad (n\in \mathbb{N}, a\in K);
\end{equation}
 \item[(3)] if $\mathrm{ord}(\mathscr{X})  = 2$, them
 \begin{equation}
 \label{eq:9}
A=\delta_2^2y - a \delta_2 y,\quad (a\in K);
 \end{equation}
\item[(4)] if $\mathrm{ord}(\mathscr{X}) = 3$, then
\begin{equation}
\label{eq:10}
A= 2 (\delta_2 y) (\delta_2^3 y) - 3 (\delta_2^2 y)^2 - a (\delta_2 y)^2,\quad (a\in K).
\end{equation}
\end{enumerate}
\end{thm}

From Theorem \ref{th:1}, when the order of a differential operator $\mathscr{X}$ is finite, an essential expansion of $\mathscr{X}$ is given by $\Lambda = \{X,A\}$, with $A\in K\{y\}$ given by \eqref{eq:7}-\eqref{eq:10}. Discussions in \cite[Chapter 2]{Ritt:50} have shown that the system of equations
\begin{equation}
\mathscr{X} y  = 0,\quad \mathscr{D}_A y = 0
\end{equation}
has solution in some extension file of $K$.  It is easy to see that this solution gives a first integral of the polynomial differential equation \eqref{eq:3}. Following result for the classification of \eqref{eq:3} is straightforward from Theorem \ref{th:1}.
\begin{thm}
\label{th:3}
Consider the polynomial differential equation \eqref{eq:3}, and let $\mathscr{X}$  the corresponding  differential operator given by \eqref{eq:1}, we have the following
\begin{enumerate}
\item[(1)] if $\mathrm{ord}(\mathscr{X}) = 0$, then \eqref{eq:3} has a first integral $\omega\in K$;
\item[(2)] if $\mathrm{ord}(\mathscr{X}) = 1$, then \eqref{eq:3} has a first integral $\omega$, such that
$$(\delta_2\omega)^n \in K$$
for some $n\in \mathbb{N}$;
\item[(3)] if $\mathrm{ord}(\mathscr{X}) = 2$, then \eqref{eq:3} has a first integral $\omega$, such that
$$\delta_2^2\omega/\delta_2\omega \in K;$$
\item[(4)] if $\mathrm{ord}(\mathscr{X}) = 3$, then \eqref{eq:3} has a first integral $\omega$, such that
$$\dfrac{2(\delta_2\omega) (\delta_2^3\omega) - 3 (\delta_2^2\omega)^2}{(\delta_2 \omega)^2} \in K;$$
\item[(5)] if $\mathrm{ord}(\mathscr{X}) = \infty$, then any first integral of \eqref{eq:3} does not satisfy any differential equation of form
$$\mathscr{D}_Ay =0$$
with $A\in K\{y\}\backslash\{X\}$.
\end{enumerate}
\end{thm}

In 1992, Singer have proved that the first three cases in Theorem \ref{th:3} (also refer Theorem \ref{th:2} below) are the only cases to have  Liouvillian integrals, i.e., there is a first integral that can be obtained from rational functions using finite steps of exponentiation, integration, an algebraic functions \cite{Singer:92}(also refer \cite{Guan:02}).  In the latter two cases,  however, the first integral of \eqref{eq:3} can not be obtained in finite steps  by the above operations from rational functions (refer \cite{Singer:92} or \cite{Guan:02}).  From the proof of Lemma \ref{lem:a.15} given below, when $\mathrm{ord}(\mathscr{X}) = 3$, the first integral of \eqref{eq:3} can be obtained through finite step operations from rational functions and a solution of the partial differential equation of form \eqref{eq:a.4}.

In the rest of this paper, we will first give the proof of Theorem \ref{th:1} in Section \ref{sec:proof}, and then give examples for each types of equations in Section \ref{sec:appl}.

\section{Proof of  the Main Result}
\label{sec:proof}
\subsection{Outline of the proof}

We always assume $X_1 \not\equiv 0$ without loss of generality.  Hereinafter, we denote $\delta_2^i y$ by $y_i$ ($y_0 = y$). For any essential expansion $\Lambda$ of $\mathscr{X}$, let $A\in \Lambda$ with the lowest rank.  From the above definitions, if $\mathrm{ord}(\mathscr{X}) = r (< \infty)$, then $A$ is a polynomial of $y_0, y_1, \cdots, y_r$, with coefficients in $K$. Write
\begin{equation}
A = \sum_\mathbf{m} a_\mathbf{m} y_0^{m_0}y_1^{m_1}\cdots y_r^{m_r},
\end{equation}
where $\mathbf{m} = (m_0, m_1,\cdots, m_r)\in \mathbb{Z}^{r+1}$, and $a_\mathbf{m}\in K$.  To prove Theorem \ref{th:1}, we only need to determine all possible non-zero coefficients in $A$.  Let
\begin{equation}
\mathcal{I}_A = \{\mathbf{m}\in \mathbb{Z}^{r+1}| a_\mathbf{m}\not=0\}.
\end{equation}
We only need to specify the finite set $\mathcal{I}_A$. The process is outlined below.

Let $\mathbf{m} = (m_0, m_1, \cdots, m_r)\in \mathbb{Z}^{r+1}$, we define an operators $\Delta_{i,j}:\mathbf{Z}^{r+1}\to\mathbf{Z}^{r+1}$ for $0<i<j\leq r$ such that  $\Delta_{i,j}(\mathbf{m})\in\mathbb{Z}^{r+1}$ is given by
\begin{equation}
\Delta_{i,j}(\mathbf{m}) = \mathbf{m} + \mathbf{e}_{j-i}  - \mathbf{e}_j
\end{equation}
where
$$\mathbf{e}_{k} = (\stackrel{\begin{array}{c}0\\ \downarrow\end{array}}{0},\cdots, 0,\stackrel{\begin{array}{c}k\\ \downarrow\end{array}}{1},0,\cdots, 0).$$
Therefore
\begin{equation}
\Delta_{i,j}^{-1}(\mathbf{m}) = \mathbf{m} - \mathbf{e}_{j-i}  + \mathbf{e}_j.
\end{equation}

For any $\mathbf{m}, \mathbf{n} \in \mathbb{Z}^{r+1}$, we will say $\mathbf{m}\succ \mathbf{n}$ if there exist $0<i<j\leq r$, such that
$$\Delta_{i,j}(\mathbf{m}) = \mathbf{n}.$$

The proof will be done by showing that if $r=\mathrm{ord}(\mathscr{X}) < \infty$, then $\mathcal{I}_A$ can only be one of the following cases:
\begin{enumerate}
\item[(1)] $r = 0$, and $\mathcal{I}_A = \{(1), (0)\}$; or
\item[(2)] $r = 1$, and $\mathcal{I}_A = \{(0,n), (0,0)\}$; or
\item[(3)] $r = 2$, and
$\mathcal{I}_A = \{(0,0,1), (0,1,0)\}$, with
$$(0, 0,1) \succ (0,1,0);$$
or
\item[(4)] $r = 3$, and $\mathcal{I}_A = \{(0,1,0,1), (0, 0,2,0), (0, 2,0,0)\}$, with relations
\begin{center}
\unitlength=0.5cm
\begin{picture}(8,4.5)
\put(2,0){(0,2,0,0)} \put(2,2){(0,1,1,0)} \put(2,4){(0,1,0,1)}
\put(5.5,2){(0,0,2,0)} \put(2.8,1){$\curlyvee$}
\put(2.8,3){$\curlyvee$} \put(4.8,2){$\prec$}
\put(1.0,4.2){\line(1,0){0.8}} \put(1.0,4.2){\line(0,-1){1.2}}
\put(1.0,0.2){\line(0,1){1.2}}\put(1.0,0.2){\line(1,0){0.8}}
\put(0.8,2.0){$\curlyvee$}
\end{picture}
\end{center}
Here $(0,1,1,0)$ is an auxiliary index with $a_{(0,1,1,0 )} = 0$.
\end{enumerate}

The final proof will be done  after 14 preliminary Lemmas, following the flow chart given in Figure \ref{fig:1}.

\begin{figure}[hbtp]
\centering
\unitlength=1.3em
\begin{picture}(24,9)
\put(18,8){\framebox{Lemma \ref{lem:a.1}}}
\put(12,8){\framebox{Lemma \ref{lem:a.2}}}
\put(6,8){\framebox{Lemma \ref{lem:a.6}}}
\put(0,8){\framebox{Lemma \ref{lem:a.9}}}
\put(0,6){\framebox{Lemma \ref{lem:a.8}}}
\put(0,2){\framebox{Lemma \ref{lem:c}}}
\put(0,4){\framebox{Lemma \ref{lem:a.10}}} 
\put(6,4){\framebox{Lemma \ref{lem:a.3}}} 
\put(18,4){\framebox{Lemma \ref{lem:a.4}}} 
\put(18,2){\framebox{Lemma \ref{lem:a.15}}}
\put(6,2){\framebox{Lemma \ref{lem:a.0}}}
\put(12,4){\framebox{Lemma \ref{le:0}}}
\put(12,2){\framebox{Lemma \ref{le:app}}}
\put(18,6){\framebox{Lemma \ref{lem:a.11}}} 
\put(-0.7,0){\framebox[31.2em]{Theorem \ref{th:1}}}

\put(17.9,8.2){\vector(-1,0){1.75}} 
\put(11.9,8.2){\vector(-1,0){1.75}}
\put(5.9,8.2){\vector(-1,0){1.75}} 
\put(2,7.6){\vector(0,-1){0.8}}
\put(4.2,6.2){\vector(1,0){13.6}} 
\put(2,5.6){\vector(0,-1){0.8}}
\put(2,1.6){\vector(0,-1){0.8}}
\put(5.8,4.2){\vector(-1,0){1.75}} 
\put(5.8,2.2){\vector(-1,1){1.75}} 
\put(20,4.8){\vector(0,1){0.8}} 
\put(20,1.6){\vector(0,-1){0.8}}  
\put(16.2,2.2){\vector(1,0){1.75}}
\put(2,3.8){\vector(0,-1){1.0}} 
\put(11.9,4.25){\vector(-1,0){1.88}}
\put(15.9,4.25){\vector(1,0){1.9}}
\put(11.9,4.2){\vector(-1,-1){1.85}}
\put(15.9,4.2){\vector(1,-1){1.9}}
\put(-0.5,6.2){\vector(0,-1){5.3}}
\put(8,1.6){\vector(0,-1){0.8}} 
\put(0,6.2){\line(-1,0){0.5}} 
\put(5.8,4.0){\line(-1,0){0.5}} 
\put(5.3,4.0){\vector(0,-1){3.2}} 
\put(22.1,4.2){\line(1,0){0.5}}
\put(22.6,4.2){\vector(0,-1){3.3}}
\put(22.1,6.2){\line(1,0){1.0}}
\put(23.1,6.2){\vector(0,-1){5.3}}

\end{picture}
\caption{Flow chart of the proof of Theorem \ref{th:1}}
\label{fig:1}
\end{figure}
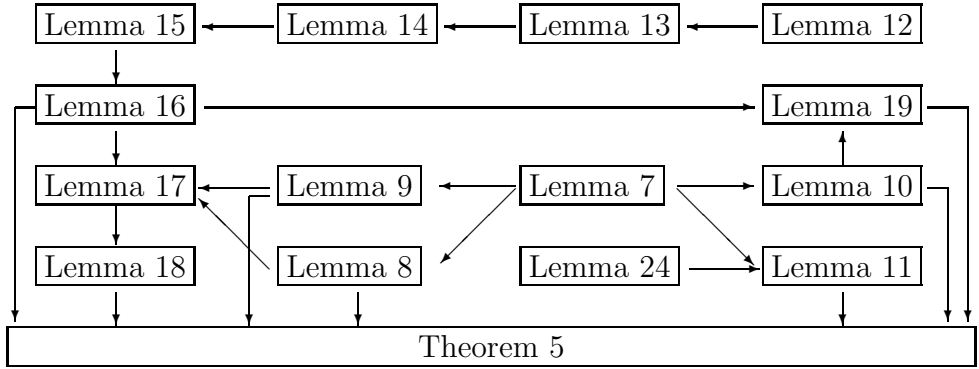

\subsection{Preliminary notations}

Before proving Theorem \ref{th:1}, we introduce some 
notations as following. 

Let
$$[\delta_2, \mathscr{X}] = \delta_2 \mathscr{X} -\mathscr{X}\delta_2 =  (\delta_2X_1) \delta_1 + (\delta_2X_2) \delta_2,$$
$$b_0 = -X_1\,(\delta_2\frac{X_2}{X_1}),\quad b_i = X_1\,(\delta_2\frac{b_{i-1}}{X_1}) = -X_1\,(\delta_2^{i+1}\frac{X_2}{X_1}),\ \  i = 1,2,\cdots$$
For $F\in K\{y\}$, and $\{X\}$ be the differential ideal that is generated by $X=\mathscr{X}y$, we write
\begin{equation}
F\sim R
\end{equation}
if $R\in K\{y\}$ such that $F - R\in\{X\}$.

Let $\mathbf{m}, \mathbf{n}\in {\mathbb{Z}}^{r+1}$, the \textit{degree} of
$\mathbf{n}$ is higher than that of $\mathbf{m}$, denoted by
$\mathbf{n}
> \mathbf{m}$, if there exists $0\leq k\leq r$ such that $n_k > m_k$ and
$$n_i = m_i,\ \ i = k+1, \cdots, r.$$

It is easy to verify that the relation $\succ$ implies $>$, and for any $\mathbf{m}\in \mathbb{Z}^{r+1}$ and $0 < i < j \leq r$,
\begin{equation}
\label{eq:a.12} \Delta_{i,j}^{-1}(\mathbf{m}) \succ \mathbf{m} \succ \Delta_{i,j}(\mathbf{m}),
\end{equation}
and
\begin{equation}
\label{eq:a.11} \Delta_{i,j}^{-1}(\mathbf{m}) > \mathbf{m} > \Delta_{i,j}(\mathbf{m}).
\end{equation}

In the following discussion, by $\mathbf{m}^*$ we will always denote the element in $\mathcal{I}_A$ with the highest
degree, and always assume $A_{\mathbf{m}^*} = 1$ without loss of generality. This is possible as the coefficients in $A$ are rational functions
in $K$. 

For any $\mathbf{m}\in \mathbb{Z}^{r+1}$, define
\begin{equation}
\label{eq:a.13} \mathcal{P}(\mathbf{m}) = \{\mathbf{p}\in \mathcal{I}_A\ |\
\mathbf{p}\succ \mathbf{m}\ \mathrm{for\ some}\ 0< i<j\leq r\}
\end{equation}
 and
$\#(\mathbf{m}) = |\mathcal{P}(\mathbf{m})|$. 

We define a function $C:\mathbb{Z}^{r+1}\to \mathbb{Z}$ by
\begin{equation}
\label{eq:a.5} C({\mathbf{m}}) = \sum_{j=1}^r jm_j,
\end{equation}
where $\mathbf{m} = (m_0, m_1, \cdots, m_r)\in \mathbb{Z}^{r+1}$.  It is easy to verify that if
$\mathbf{m}\succ \mathbf{p}$, then $C(\mathbf{m}) >
C(\mathbf{p})$. In particularly,  
\begin{equation}
\label{eq:a.c}
 C(\mathbf{m}) - C(\Delta_{i,j}(\mathbf{m})) = i,\quad (0<i<j\leq r).
\end{equation}

\subsection{Preliminary Lemmas}

Now, we can start the proof process. First, following lemma is straightforward from the definition of nontrivial expansion.
\begin{lem}
\label{le:0}
Let $A\in K\{y\}$, the differential ideal $\Lambda = \{A, X\}$ is a nontrivial expansion of $\mathscr{X}$ if, and only if, the equation
\begin{equation}
\label{eq:xa}
\left\{
\begin{array}{rcl}
\mathscr{X}y &=& 0\\
\mathscr{D}_Ay &=& 0
\end{array}\right.
\end{equation}
has a non constant solution in $\mathcal{A}(\Omega)$, with $\Omega$ an open subset of $\mathbb{C}^2$.
\end{lem}

Following result is a straightforward conclusion from Lemma \ref{le:0}
\begin{lem}
\label{lem:a.0}
If there exist $a\in K$, non constant, such that $\mathscr{X}a  = 0$, then let
$$A = y - a,$$
the differential ideal $\Lambda = \{X,A\}$ is a nontrivial expansion of $\mathscr{X}$.
\end{lem}

\begin{lem}
\label{lem:a.3} If there exists $a\in K$, $a\not=0$, such that
\begin{equation}
\label{eq:a.8} \mathscr{X} a  = nb_0 a,
\end{equation}
where $n$ is non-zero integer, let 
\begin{equation}
A = (\delta_2y)^{|n|} - a^{|n|/n},
\end{equation}
then $\Lambda = \{X, A\}$ is a nontrivial expansion of $\mathscr{X}$.
\end{lem}
\begin{proof} From Lemma \ref{le:0}, we only need to show that there is a non constant solution of the differential equation
\begin{equation}
\label{eq:e1}
\left\{
\begin{array}{rcl}
X_1 \delta_1 y + X_2 \delta_2 y &=& 0\\
(\delta_2y)^{|n|} - a^{|n|/n} &=& 0.
\end{array}\right.
\end{equation} 

Let
$$u = a^{1/n},\quad v = -\dfrac{X_2}{X_1} u,$$
and taking account \eqref{eq:a.8}, direct calculations show that 
$$\delta_1 u = \dfrac{1}{X_1} (b_0 u - X_2 \delta_2 u) = \delta_2 v.$$
Thus, the  1-form $v dx_1 + u d x_2$ is closed, and therefore the function of form
$$\omega = \int_{(x_1^0, x_2^0)}^{(x_1,x_2)} v dx_1 + u dx_2$$
is well defined and analytic on a neighborhood of some $(x_1^0, x_2^0)\in \mathbb{C}^2$.  Further, 
$$\delta_1 \omega= v,\quad \delta_2 \omega= u.$$
It is easy to verify that $\omega$ satisfies equations \eqref{eq:e1}, and the Lemmas is proved.
\end{proof}

\begin{lem}
\label{lem:a.4}If there exists $a\in K$ satisfying
\begin{equation}
\label{eq:a.9} \mathscr{X} a = b_0 a  + b_1.
\end{equation}
Let
\begin{equation}
A = \delta_2^2y - a \delta_2y,
\end{equation}
then $\Lambda = \{X,A\}$ is a nontrivial expansion of $\mathscr{X}$.
\end{lem}
\begin{proof}  Let 
$$b  = -\dfrac{X_2}{X_1} a + \dfrac{b_0}{X_1}.$$
From \eqref{eq:a.9}, we have
$$\delta_1 a = -\dfrac{X_2}{X_1} \delta_2 a - (\delta_2\dfrac{X_2}{X_1}) a + \delta_2 \dfrac{b_0}{X_1} = \delta_2 b.$$
Thus, the 1-form $bdx_1 + a dx_2$ is closed, and there exists a function $\eta$ that is  analytic on a neighborhood of some $(x_1^0,x_2^0)\in \mathbb{C}^2$, such that
$$\delta_1 \eta = b, \quad \delta_2 \eta = a.$$
Furthermore, we assume that $X_1(x_1^0,x_2^0)\not=0$. Let 
$u =  \exp(\eta)$, then $u$ is a non zero function, and 
$$
\mathscr{X}u = u (X_1 \delta_1 \eta + X_2 \delta_2 \eta)=u (X_1 b + X_2 a)  = b_0 u.
$$
Thus, following the proof of Lemma \ref{lem:a.3}, let
$$v = -\dfrac{X_2}{X_1}u,$$
then $v d x_1 + u d x_2$ is a closed 1-form, and the function 
$$\omega = \int_{(x_1^0,x_2^0)}^{(x_1,x_2)} vd x_1 + u dx_2$$
is well defined in a neighborhood of $(x_1^0,x_2^0)$ (we note that $X_1(x_1^0,x_2^0) \not=0$), non constant, and satisfies
$$X_1\delta_1\omega +X_2 \delta_2 \omega = 0,\quad \delta_2 \omega - u  = 0.$$
Therefore, 
$$X_1\delta_1\omega +X_2 \delta_2 \omega = 0, \quad \delta_2^2\omega - a \delta_2 u = 0.$$
Thus, the non constant function $\omega$ satisfies the equation
\begin{equation}
\left\{
\begin{array}{rcl}
X_1 \delta_1 y + X_2 \delta_2 y &=& 0 \\
\delta_2^2 y - a \delta_2 y &=& 0
\end{array}
\right.
\end{equation}
and hence the Lemma is concluded from Lemma \ref{le:0}.
\end{proof}

\begin{lem}
\label{lem:a.15}If there exists $a\in K$ satisfying
\begin{equation} 
\label{eq:a.10} 
\mathscr{X} a = 2 b_0 a  + b_2.
\end{equation}
Let
\begin{equation}
A = 2 (\delta_2 y) (\delta_2^3 y) - 3 (\delta_2^2 y)^2 - a (\delta_2 y)^2,
\end{equation}
then $\Lambda = \{X,A\}$ is a nontrivial expansion of $\mathscr{X}$.
\end{lem}
\begin{proof}
We will show that there is a function $\omega$ that is analytic on an open subset of $\mathbb{C}^2$, non constant,  and satisfies
\begin{equation}
\label{eq:o3}
\left\{
\begin{array}{l}
X_1 \delta_1 \omega + X_2 \delta_2 \omega = 0\\
2 (\delta_2 \omega) (\delta_2^3 \omega) - 3 (\delta_2^2 \omega)^2 - a (\delta_2 \omega)^2 = 0.
\end{array}\right.
\end{equation}

Let
\begin{eqnarray*} 
f(x_1,x_2,u) &=& -\delta_2^2 \dfrac{X_2}{X_1} - \dfrac{X_2}{X_1} a - (\delta_2 \dfrac{X_2}{X_1}) u - \dfrac{1}{2} (\dfrac{X_2}{X_1}) u^2,\\
g(x_1,x_2, u) &=& a + \dfrac{1}{2} u^2.
\end{eqnarray*}
Then $f$ and $g$ are analytic at some point $(x_1^0, x_2^0, u^0)\in \mathbb{C}^3$. We further assume that $X_1(x_1^0, x_2^0)\not=0$. We will show that there is a function $u(x_1,x_2)$ that is analytic on a neighborhood of $(x_1^0,x_2^0)$,  and $u(x_1^0, x_2^0) = u^0$, such that 
\begin{equation}
\label{eq:a.4}
\left\{
\begin{array}{rcl}
\delta_1 u &=& f(x_1,x_2,u),\\
\delta_2 u &=& g(x_1,x_2,u)
\end{array}\right.
\end{equation}
is satisfied in a neighborhood of $(x_1^0,x_2^0)$.

It follows from \eqref{eq:a.10} that
$$\delta_1a = - \delta_2^3 \dfrac{X_2}{X_1} - \delta_2 (\dfrac{X_2}{X_1} a) - (\delta_2 \dfrac{X_2}{X_1}) a .$$
Thus, from the \eqref{eq:a.4}, we have
\begin{eqnarray*}
&&(\dfrac{\partial\ }{\partial x_2} + g(x_1,x_2,u)\dfrac{\partial\ }{\partial u} ) f(x_1,x_2,u)\\
&=&-\delta_2^3\frac{X_2}{X_1} - \delta_2(\frac{X_2}{X_1}a) -
(\delta_2^2\frac{X_2}{X_1}) u - \frac{1}{2}(\delta_2\frac{X_2}{X_1}) u^2\\ 
&&{} - g(x_1,x_2,u) (\delta_2\frac{X_2}{X_1}  +\frac{X_2}{X_1}u) \\
&=&-\delta_2^3\frac{X_2}{X_1} - \delta_2(\frac{X_2}{X_1}a) - (\delta_2\frac{X_2}{X_1}) a -
(\delta_2^2\frac{X_2}{X_1}) u - (\frac{X_2}{X_1}a) u\\
&&{} - (\delta_2\frac{X_2}{X_1}) u^2 - \frac{1}{2}(\frac{X_2}{X_1}) u^3,\\
&=& -\delta_2^3\frac{X_2}{X_1} - \delta_2(\frac{X_2}{X_1}a) - (\delta_2\frac{X_2}{X_1}) a + u f(x_1,x_2,u)\\
&=&\delta_1 a + u f(x_1,x_2,u)\\
&=&(\dfrac{\partial\ }{\partial x_1} + f(x_1,x_2,u)\dfrac{\partial\ }{\partial u}) g(x_1,x_2,u).
\end{eqnarray*}
Therefore, assuming
\begin{equation}
\label{eq:us}
u(x_1,x_2) = \sum_{i=0}^\infty\sum_{j=0}^\infty u_{i,j} (x_1-x_1^0)^i (x_2-x_2^0)^j,\quad (u_{0,0} = u^0)
\end{equation}
and applying the Method of Majorants, we can obtain the coefficients $u_{i,j}$ by induction, and the power series \eqref{eq:us} is convergent in a neighborhood of $(x_1^0,x_2^0)$ (refer Appendix for detail), which yields an analytic solution of \eqref{eq:a.4}.

Let $u$ to be the above solution of \eqref{eq:a.4}, and 
$$v = -\delta_2\frac{X_2}{X_1} - \frac{X_2}{X_1} u.$$
It is easy to verify $\delta_2 v = \delta_1 u$, and hence 
the 1-form $v dx_1 + u dx_2$ is closed. Let
\begin{equation}
\label{eq:eta}
\eta =
\exp\left[\int_{(x_1^0,x_2^0)}^{(x_1,x_2)} v dx_1 + u dx_2\right],
\end{equation}
then the function $\eta$ is well defined, non zero,  and analytic on a neighborhood of $(x_1^0, x_2^0)$ (here we note that $X_1(x_1^0, x_2^0) \not=0$), and
$$\mathscr{X}\eta = b_0 \eta.$$
Follow the proof of Lemma \ref{lem:a.3}, there exist a non constant function $\omega$, analytic on an a neighborhood of $(x_1^0, x_2^0)$ (we note that $b_0$ is analytic at $(x_1^0,x_2^0)$), such that
$$\mathscr{X}\omega = 0,\quad \delta_2 \omega = \eta.$$

From \eqref{eq:eta} and \eqref{eq:a.4}, we have $\delta_2 \eta = \eta u$, and
$$\eta \delta_2^2 \eta = \eta \left((\delta_2 \eta) u + \eta (a + \dfrac{1}{2} u^2)\right) = \dfrac{3}{2}(\delta_2 \eta)^2 + a \eta^2.$$
Taking account $\eta = \delta_2 \omega$, we have
$$(\delta_2 \omega) (\delta_2^3 \omega) - \dfrac{3}{2} (\delta_2^2 \omega)^2 - a (\delta_2 \omega)^2 = 0.$$
Thus, $\omega$ satisfies \eqref{eq:o3} and the Lemma is concluded.  
\end{proof}

\begin{lem}
\label{lem:a.1} Let $[\delta_2, \mathscr{X}]$ and $y_i$ defined as previous, then 
\begin{enumerate} 
\item[(1)] $[\delta_2, \mathscr{X}] =
(\frac{\delta_2 X_1}{X_1}) \mathscr{X} - b_0 \delta_2$; 
\item[(2)] $\mathscr{X}y_j =
\delta_2\mathscr{X}y_{j-1} - (\frac{\delta_2 X_1}{X_1}) \mathscr{X}y_{j-1} +
b_0\,y_j$.
\end{enumerate}
\end{lem}
\begin{proof}
(1) is straightforward from
\begin{eqnarray*} [\delta_2,\mathscr{X}] &=& (\delta_2X_1) \delta_1 +
(\delta_2X_2) \delta_2\\
&=&\frac{\delta_2 X_1}{X_1} (X_1\,\delta_1 + X_2 \delta_2) -
\frac{X_2}{X_1} (\delta_2 X_1) \delta_2 + (\delta_2
X_2)\,\delta_2\\
&=&\frac{\delta_2 X_1}{X_1} \mathscr{X} + X_1\,\frac{X_1 \delta_2X_2 -
X_2\,\delta_2 X_1}{X_1^2}\delta_2\\
&=&\frac{\delta_2X_1}{X_1} \mathscr{X} - b_0\,\delta_2.
\end{eqnarray*}
(2) can be obtained by direct calculation as follows:
\begin{eqnarray*}
\mathscr{X}y_j &=& \mathscr{X}\delta_2 y_{i-1}\\
&=&\delta_2\mathscr{X}y_{j-1} - [\delta_2, \mathscr{X}]y_{j-1}\\
&=& \delta_2\mathscr{X}y_{j-1} - (\frac{\delta_2X_1}{X_1}\,\mathscr{X} -
b_0 \delta_2)y_{j-1}\\
&=& \delta_2\mathscr{X}y_{j-1} - \frac{\delta_2X_1}{X_1}\,\mathscr{X}y_{j-1} +
b_0 \delta_2y_{j-1}\\
&=&\delta_2\mathscr{X}y_{j-1} - (\frac{\delta_2 X_1}{X_1})\,\mathscr{X}y_{j-1} +
b_0 y_j.
\end{eqnarray*}
\end{proof}

\begin{lem}
\label{lem:a.2} We have
\begin{equation} \label{eq:a.14} 
\mathscr{X}y_j \sim
\sum_{i = 0}^{j-1}c_{i,j}\,b_i\,y_{j-i},\quad (j\geq 1)
\end{equation}
 where $c_{i,j}$ are positive integers, and $c_{0,j} = j$.
\end{lem}
\begin{proof} From Lemma \ref{lem:a.1}, when $j = 1$, we have
$$
\mathscr{X}y_1 = \delta_2\mathscr{X}y_0 - (\frac{\delta_2 X_1}{X_1})\mathscr{X}y_0 + b_0 y_1 \sim b_0 y_1.
$$
Thus \eqref{eq:a.14} holds for $j=1$ with $c_{0,1} = 1$. 

Assume that \eqref{eq:a.14} is valid for $j = k$ with positive integer coefficients $c_{i,k}$, and $c_{0,k}=k$, applying
 Lemma \ref{lem:a.1}, we have
\begin{eqnarray*}
\mathscr{X}y_{k+1} &=& \delta_2\mathscr{X}y_k - (\frac{\delta_2
X_1}{X_1})\,\mathscr{X}y_k + b_0\,y_{k+1}\\
&\sim&\delta_2(\sum_{i=0}^{k-1}c_{i,k}b_i y_{k-i}) -
(\frac{\delta_2 X_1}{X_1}) (\sum_{i=0}^{k-1}c_{i,k}b_i y_{k-i}) + b_0\,y_{k+1}\\
&=&\sum_{i=0}^{k-1}c_{i,k} ((\delta_2b_i) y_{k-i} + b_i\,\delta_2y_{k-i}) -
\sum_{i=0}^{k-1}c_{i,k}\frac{\delta_2 X_1}{X_1} b_i y_{k-i} + b_0y_{k+1}\\
&=&\sum_{i=0}^{k-1}c_{i,k}\left((\delta_2b_i - \frac{\delta_2
X_1}{X_1} b_i) y_{k-i} + b_i y_{k-i+1}\right) + b_0\,y_{k+1}\\
&=&(c_{0,k} + 1) b_0\,y_{k+1} + \sum_{i = 0}^{k-2}\left(c_{i,k}
X_1 \delta_2(\frac{b_i}{X_1}) + c_{i+1,k}\,b_{i+1}\right)y_{k-i}\\
&&{} + c_{k-1,k} X_1\delta_2(\frac{b_{k-1}}{X_1}) y_1\\
&=&(c_{0,k} + 1) b_0 y_{k+1} + \sum_{i = 0}^{k-2}(c_{i, k}
 + c_{i+1, k}) b_{i+1}y_{k-i} + c_{k-1, k} b_{k} y_1.
\end{eqnarray*}
Thus, let
$$\left\{\begin{array}{ll}
c_{0,k+1} = c_{0,k} + 1 = k+1,& \\
c_{i,k+1} = c_{i-1,k} + c_{i,k},&\ \ (1\leq i\leq k-1),\\
c_{k,k+1} = c_{k-1,k},&
\end{array}\right.$$
which are positive integers,  we have 
$$\mathscr{X}y_{k+1} \sim  \sum_{i = 0}^kc_{i,k+1} b_i y_{k + 1 - i},$$
and the  Lemma is proved.
\end{proof}

\begin{lem}
\label{lem:a.6} We have
\begin{equation}
\mathscr{X}a_{\mathbf{m}}\mathbf{y}^\mathbf{m}\sim (\mathscr{X}a_{\mathbf{m}}  + C({\mathbf{m}}) b_0
a_{\mathbf{m}})\mathbf{y}^{\mathbf{m}} + \sum_{i =
1}^{r-1}\sum_{j=i+1}^rm_jc_{i,j}b_ia_\mathbf{m}\mathbf{y}^{\Delta_{i,j}(\mathbf{m})},
\end{equation}
where $a_\mathbf{m}\in K$,  $\mathbf{y}^{\mathbf{m}} = y_0^{m_0}y_1^{m_1}\cdots y_r^{m_r}$, and $c_{i,j}$ is defined as in Lemma \ref{lem:a.2}.
\end{lem}
\begin{proof} It is easy to have
$$\mathscr{X}a_{\mathbf{m}}\mathbf{y}^{\mathbf{m}} = (\mathscr{X}a_{\mathbf{m}}) \mathbf{y}^{\mathbf{m}}+ a_{\mathbf{m}}\sum_{j=0}^r \dfrac{\partial \mathbf{y}^{\mathbf{m}}}{\partial y_j}\mathscr{X}y_j.$$
From Lemma \ref{lem:a.2}, we have
\begin{eqnarray*}
\mathscr{X}a_{\mathbf{m}}\mathbf{y}^{\mathbf{m}} &=&(\mathscr{X}a_\mathbf{m})\mathbf{y}^{\mathbf{m}} + a_\mathbf{m} \sum_{j=0}^r m_j \mathbf{y}^{\mathbf{m}-\mathbf{e}_j} \mathscr{X}y_i\\
&\sim&(\mathscr{X}a_\mathbf{m})\mathbf{y}^{\mathbf{m}} + a_\mathbf{m}\,\sum_{j=1}^r m_j \mathbf{y}^{\mathbf{m}-\mathbf{e}_j}(\sum_{i=0}^{j-1}c_{i,j} b_i y_{j-i})\\
&=&(\mathscr{X}a_\mathbf{m})\mathbf{y}^{\mathbf{m}} +
a_\mathbf{m} b_0(\sum_{j = 1}^rc_{0, j} m_j)\mathbf{y}^{\mathbf{m}} + a_\mathbf{m}\sum_{j=1}^r\sum_{i = 1}^{j-1}m_jc_{i,j}b_i \mathbf{y}^{\mathbf{m} + \mathbf{e}_{j-i} -\mathbf{e}_j}\\
&=&(\mathscr{X} a_{\mathbf{m}} + C({\mathbf{m}})\, b_0 a_{\mathbf{m}})\,\mathbf{y}^{\mathbf{m}} + \sum_{i =
1}^{r-1}\sum_{j=i+1}^r m_j c_{i,j} b_i a_\mathbf{m}\,\mathbf{y}^{\Delta_{i,j}(\mathbf{m})},
\end{eqnarray*}
and the Lemma is concluded.
\end{proof}

\begin{lem}
\label{lem:a.9} Let $\Lambda$ be a nontrivial expansion of $\mathscr{X}$, $A\in \Lambda$ with the lowest rank and $r =
\mathrm{ord}(\Lambda) $. Let $\mathbf{m}^*\in \mathcal{I}_A$ with the
highest degree and assume that $a_{\mathbf{m}^*} = 1$, then for any $\mathbf{m}\in \mathbb{Z}^{r+1}$, $\mathbf{m}<\mathbf{m}^*$, we have
\begin{equation}
\label{eq:a.2} \mathscr{X}a_\mathbf{m} = (C({\mathbf{m}^*}) -
C({\mathbf{m}}))b_0a_{\mathbf{m}}
-\sum_{i=1}^{r-1}\sum_{j=i+1}^{r}(m_j+1)c_{i,j}b_ia_{{\Delta_{i,j}^{-1}}(\mathbf{m})}.
\end{equation}
Here $a_{\mathbf{m}} = 0$ whenever  $\mathbf{m}\not\in \mathcal{I}_A$.
\end{lem}
\begin{proof}
We can write 
$$A = \sum_{\mathbf{m}\in \mathcal{I}_A}a_{\mathbf{m}}\mathbf{y}^{\mathbf{m}} = \sum_{\mathbf{m} \leq \mathbf{m}^*} a_{\mathbf{m}} \mathbf{y}^{\mathbf{m}}.$$
Hereinafter $a_{\mathbf{m}} = 0$ if $\mathbf{m}\not\in \mathcal{I}_A$. 

First, it is easy to have 
$$\mathscr{X}A = X_1 \delta_1A + X_2 \delta_2 A \in \Lambda.$$ 
On the other hand, from Lemma \ref{lem:a.6}, we have
\begin{eqnarray*}
\mathscr{X}A&=&\sum_{\mathbf{m}\leq \mathbf{m}^*}\mathscr{X}a_{\mathbf{m}}\mathbf{y}^{\mathbf{m}}\\
&\sim&\sum_{\mathbf{m}\leq \mathbf{m}^*}\left((\mathscr{X}a_{\mathbf{m}} +
C({\mathbf{m}})b_0a_{\mathbf{m}})\mathbf{y}^{\mathbf{m}}  +
\sum_{i=1}^{r-1}\sum_{j = i+1}^{r}m_jc_{i,j} b_i a_{\mathbf{m}}\mathbf{y}^{\Delta_{i,j}(\mathbf{m})}\right)\\
&=&\sum_{\mathbf{m}\leq \mathbf{m}^*}\left(\mathscr{X}a_{\mathbf{m}} +
C({\mathbf{m}})b_0a_{\mathbf{m}} + \sum_{i=1}^{r-1}\sum_{j =
i+1}^{r}(m_j+1)c_{i,j}b_ia_{{\Delta_{i,j}^{-1}}(\mathbf{m})}\right)\mathbf{y}^{\mathbf{m}}
\end{eqnarray*}
Note that for any $j>i$, $\Delta_{i,j}^{-1}(\mathbf{m}^*) > \mathbf{m}^*$,  and
thus $\Delta_{i,j}^{-1}(\mathbf{m}^*) \not\in \mathcal{I}_A$, i.e., $a_{\Delta_{i,j}^{-1}(\mathbf{m}^*)}
= 0$ for any $j>i$. Taking account $a_{\mathbf{m}^*} = 1$, we have $\mathscr{X}a_{\mathbf{m}^*} =0$, and hence
\begin{eqnarray*}
\mathscr{X}A &\sim& C(\mathbf{m}^*)b_0 y^{\mathbf{m}^*} \\
&&{} +  \sum_{\mathbf{m}<\mathbf{m}^*}\left(\mathscr{X}a_{\mathbf{m}} +
C({\mathbf{m}})b_0a_{\mathbf{m}} + \sum_{i=1}^{r-1}\sum_{j =
i+1}^{r}(m_j+1)c_{i,j}b_ia_{{\Delta_{i,j}^{-1}}(\mathbf{m})}\right)\mathbf{y}^{\mathbf{m}}.
\end{eqnarray*}
Therefore, 
\begin{equation}
\label{eq:l58}
\mathscr{X}A - C({\mathbf{m}^*})b_0A  \sim R = \sum_{\mathbf{m} < \mathbf{m}^*} f_{\mathbf{m}} \mathbf{y}^{\mathbf{m}}, 
\end{equation}
where the coefficients $f_{\mathbf{m}}$ are
\begin{equation}
\label{eq:l58-1}
f_{\mathbf{m}} =  \mathscr{X}a_\mathbf{m}
+ (C({\mathbf{m}}) - C({\mathbf{m}^*}))b_0a_{\mathbf{m}} +
\sum_{i=1}^{r-1}\sum_{j=i+1}^{r-1}(m_j+1)c_{i,j}b_ia_{{\Delta_{i,j}^{-1}}(\mathbf{m})}.
\end{equation}

Now, we obtain a differential polynomial $R$ that has lower rank than $A$ and is contained in the differential ideal
$\Lambda$. But $A$ is an element in $\Lambda$ with the lowest rank. Thus, we must have  
$R\equiv 0$. Therefore the coefficients \eqref{eq:l58-1} are zero, from which  \eqref{eq:a.2} is concluded. The Lemma has been proved.
\end{proof}

Note that $a_{\mathbf{m}^*} = 1$ and $\Delta_{i,j}^{-1}(\mathbf{m}^*) \not\in\mathcal{I}_A$,  the equation \eqref{eq:a.2} is also valid for $a_{\mathbf{m}^*}$.

The equation \eqref{eq:a.2} can be rewritten in another form as follows.
\begin{lem}
\label{lem:a.8} In Lemma \ref{lem:a.9}, for any $\mathbf{m}\leq \mathbf{m}^*$, let $k=\#(\mathbf{m})$ and $\mathcal{P}(\mathbf{m}) = \{\mathbf{p}_1,\cdots,
\mathbf{p}_k\}$, and assume $\Delta_{{i_l},{j_l}}(\mathbf{p}_l) = \mathbf{m},\
(l = 1,2,\cdots,k)$, then the coefficients $a_{\mathbf{p}_l},
a_{\mathbf{m}}$ satisfy
\begin{equation}
\label{eq:a.3} \mathscr{X} a_{\mathbf{m}} = (C({\mathbf{m}^*}) -
C({\mathbf{m}}))b_0a_{\mathbf{m}} -\sum_{l = 1}^{\#(\mathbf{m})}
(m_{j_l}+1) c_{i_l,j_l} b_{i_l} a_{\mathbf{p}_l}.
\end{equation}
\end{lem}

\begin{lem}
\label{lem:a.10}
Let $\Lambda$ be a nontrivial expansion of $\mathscr{X}$, $A\in \Lambda$ with the lowest rank and $r =
\mathrm{ord}(\Lambda) > 1$. Let $\mathbf{m}^*\in \mathcal{I}_A$ with the highest
degree. Then for any $\mathbf{m}\in \mathcal{I}_A$, $\#(\mathbf{m})=0$ if
and only if $C(\mathbf{m}) = C(\mathbf{m}^*)$. Furthermore, if
$\#(\mathbf{m})=0$, then $a_{\mathbf{m}}$ is a constant.
\end{lem}
\begin{proof} First, we will prove that if $\#(\mathbf{m}) = 0$, then $C(\mathbf{m}) = C(\mathbf{m}^*)$.

If $\#(\mathbf{m}) = 0$,  then Lemma \ref{lem:a.8} yields
$$\mathscr{X}a_{\mathbf{m}} = (C({\mathbf{m}^*}) -
C({\mathbf{m}}))b_0a_{\mathbf{m}}.$$ 
If otherwise $C({\mathbf{m}})\not=C({\mathbf{m}^*})$, then $n = C(\mathbf{m}^*) -
C(\mathbf{m})$ is a non-zero integer, and $a_{\mathbf{m}^*}\not=0$ such that
$$\mathscr{X}a_{\mathbf{m}}
 = n b_0 a_{\mathbf{m}}.$$
 From Lemma \ref{lem:a.3}, let
 $$A' = (\delta_2 y)^{|n|} - {a_{\mathbf{m}}}^{|n|/n},$$
 then the differential ideal $\Lambda' = \{X,A'\}$ is a nontrivial expansion of $\mathscr{X}$, and with order $\leq 1$.  This contradicts with the assumption that $\Lambda$ is an essential expansion with order $> 1$.  Thus, we have concluded that $C(\mathbf{m}) =
 C(\mathbf{m}^*)$.

Next, we will prove that if $C(\mathbf{m}) = C(\mathbf{m}^*)$, then $\#(\mathbf{m}) = 0$. 

If on the contrary, $C(\mathbf{m}) = C(\mathbf{m}^*)$ but $\#(\mathbf{m})>0$,  there exists $\mathbf{m}_1\in
 \mathcal{P}(\mathbf{m})$. From \eqref{eq:a.c}, we have $C(\mathbf{m}_1) > C(\mathbf{m}) =
 C(\mathbf{m}^*)$.  Apply the previous part of the proof to $\mathbf{m}_1$, we have $\#(\mathbf{m}_1) > 0$. Thus, we can repeat the above process, and obtain $\mathbf{m}_2\in \mathcal{P}(\mathbf{m}_1)$ such that $C(\mathbf{m}_2) > C(\mathbf{m}_1)>C(\mathbf{m}^*)$ and $\#(\mathbf{m}_2) > 0$.  This procedure can continue to obtain an infinite sequence $\{\mathbf{m}_k\}_{k=1}^\infty \subseteq \mathcal{I}_A$ such that $\#(\mathbf{m}_k) > 0$ and $C(\mathbf{m}_{k+1}) > C(\mathbf{m}_k)  > C(\mathbf{m}^*)$. But $\mathcal{I}_A$ is a finite set. Thus, we come to a contradiction, and therefore $\#(\mathbf{m}) = 0$.

Now, we have  proved that $\#(\mathbf{m})$ if and only if $C(\mathbf{m}) = C(\mathbf{m}^*)$.

If $\#(\mathbf{m}) = 0$, then $C(\mathbf{m}^*) = C(\mathbf{m})$, and therefore
\eqref{eq:a.3} yields $\mathscr{X}a_{\mathbf{m}} = 0$. But
$\mathrm{ord}(\Lambda)
> 1$, thus $a_{\mathbf{m}}$ is a constant according to Lemma \ref{lem:a.0}.
\end{proof}

\begin{lem}
\label{lem:c}
Let $\Lambda$ be a nontrivial expansion of $\mathscr{X}$, $A\in \Lambda$ with the lowest rank and $r =
\mathrm{ord}(\Lambda) > 1$. Let $\mathbf{m}^*\in \mathcal{I}_A$ with the highest
degree. Then for any $\mathbf{m}\in \mathcal{I}_A$, $C(\mathbf{m})\leq C(\mathbf{m}^*)$.
\end{lem}
\begin{proof}If otherwise, there is $\mathbf{m}\in\mathcal{I}_A$ such that $C(\mathbf{m}) > C(\mathbf{m}^*)$, then $\#(\mathbf{m}) \geq 1$ by Lemma \ref{lem:a.10}.  Thus, there is a $\mathbf{m}_1\in \mathcal{P}(\mathbf{m})$, and $C(\mathbf{m}_1) > C(\mathbf{m}) > C(\mathbf{m}^*)$. Thus, we can repeat the procedure to obtain an infinite sequence $\{\mathbf{m}_k\}_{k=1}^{\infty}\subseteq \mathcal{I}_A$.  This is contradiction to the fact that $\mathcal{I}_A$ is a finite set, and the Lemma is concluded.
\end{proof}

\begin{lem}
\label{lem:a.11} Assume $r= \mathrm{ord}(\mathscr{X}) \geq 3$. Let $\Lambda$ be an essential expansion of $\mathscr{X}$, $A\in \Lambda$ with the lowest rank, $\mathbf{m}^* = (m_0^*, m_1^*, \cdots,m_r^*)\in \mathcal{I}_A$ with the highest degree and $a_{\mathbf{m}^*} = 1$, then $m_1^*>0$ and $m_2^* = 0$.
\end{lem}
\begin{proof}
(1). If $m_1^* = 0$, we can write $\mathbf{m}^*$ as 
$$\mathbf{m}^* = (m_0^*, 0,\cdots, 0, m_k^*, \cdots, m_r^*),$$
where $1 < k \leq r$ and $m_k^* >0 $.
Let 
$$\mathbf{m} = \Delta_{1,k}(\mathbf{m}^*) = (m_0^*,0,\cdots,0,1,m_k^*-1,m_{k+1}^*,\cdots,m_r^*),$$
it is easy to have $\mathcal{P}(\mathbf{m}) = \{\mathbf{m}^*\}$. Hence,
$$\mathscr{X}a_{\mathbf{m}} = b_0 a_\mathbf{m} - m_k^* c_{1,k} b_1$$
from Lemma \ref{lem:a.8}. Here we have applied $C(\mathbf{m}^*) - C(\mathbf{m}) = 1$ and $a_{\mathbf{m}^*} = 1$. Let
$$a = -\dfrac{a_{\mathbf{m}}}{m_k^* c_{1,k}},$$
then 
$$\mathscr{X} a = b_0 a + b_1.$$
Thus, we have $\mathrm{ord}(\mathscr{X}) \leq 2$ from Lemma \ref{lem:a.4}, which contradicts with $r \geq 3$.

(2). If $m_2^*>0$, let 
$$\mathbf{p} = \Delta_{1,2}(\mathbf{m}^*) =
(m_0^*, m_1^*+1,m_2^*-1,m_3^*,\cdots,m_r^*).$$ 
It is easy to verify 
$\mathcal{P}(\mathbf{p}) = \{\mathbf{m^*}\}$ as follows. (1) Since $\Delta_{1,2}(\mathbf{m^*}) = \mathbf{p}$, we have $\mathbf{m^*}\in\mathcal{P}(\mathbf{p})$. (2) If there is any other $\mathbf{m'}\in \mathcal{P}(\mathbf{p})$, then $\Delta_{i,j}(\mathbf{m'})
=\mathbf{p}$ for some $(i,j)\not= (1,2)$. Thus, we always have $j>2$, which yields $\mathbf{m}' >\mathbf{m^*}$, and hence contradicts with the assumption that $\mathbf{m}^*$ is the highest.

Hence, we have
$$\mathscr{X}a_{\mathbf{p}} =  (C(\mathbf{m}^*) - C(\mathbf{p})) b_0 a_{\mathbf{p}} - c_{1,2} m_2^* b_1a_{\mathbf{m}^*}
$$
from Lemma \ref{lem:a.8}. Similar to the above argument as in (1), we have $\mathrm{ord}(\mathscr{X})\leq 2$, which is contradiction to the assumption. Thus, we must have $m_2^* = 0$.
\end{proof}

\subsection{Proof of Theorem \ref{th:1}}

Now, we are ready to prove our main Theorem.

\begin{proof}[Proof of Theorem \ref{th:1}]
Let $\Lambda$ be a nontrivial expansion of $\mathscr{X}$, and $A\in \Lambda$ with  the lowest rank, $\mathbf{m}^*\in \mathcal{I}_A$ with the highest degree, and $a_{\mathbf{m}^*} = 1$. 

(1).  If $r = 0$,  let $n = \mathbf{m}^*$,  we can write $A$ as
$$A = y^n + a_1 y^{n-1} + \cdots + a_n,\quad (a_i \in K,  i  = 1,\cdots, n),$$
with at least one $a_i\in K\backslash\mathbb{C}$. Thus, the equation \eqref{eq:a.3} implies
$$\mathscr{X}a_i = 0.$$
Let
$$B = y - a_i,$$
then $a_i$ satisfies equations
\begin{equation}
\mathscr{X}y = 0,\quad \mathscr{D}_{B}y = 0.
\end{equation}
Hence, $ \{X, B\}$ is a nontrivial expansion of $\mathscr{X}$ with  order $0$, and (1) is concluded. 

(2). If $r = 1$, we argue that there exists $\mathbf{m}\in \mathcal{I}_A$, with $\mathbf{m} < \mathbf{m}^*$, such that $C(\mathbf{m}) \not= C(\mathbf{m}^*)$. 
If otherwise, for any $\mathbf{m}\in \mathcal{I}_A$,  $C(\mathbf{m}) = C(\mathbf{m}^*)$, then $A$ must has form
$A = (\delta_2y)^{n}p(y)$, where $n = C(\mathbf{m}^*)$ and $p(y)$ is a polynomial of $y$, with coefficients in $K$.  Thus, let $\omega$ to be a non constant solution of 
$$\mathscr{X} y = 0,\quad \mathscr{D}_A y = 0,$$
then either
$$
\mathscr{X} \omega = 0,\quad \delta_2 \omega = 0,
\quad \mathrm{or}\quad \mathscr{X} \omega = 0,\quad p(\omega) = 0.
$$
But these are not possible because the former case implies $X_1\equiv 0$, and latter case implies  $\mathrm{ord}(\Lambda) = 0$, both are in contradiction to our assumptions. 

Now, let $\mathbf{m}$ such that $C(\mathbf{m}) \not = C(\mathbf{m}^*)$. We note that $\#(\mathbf{m}) = 0$, thus, the equation \eqref{eq:a.3} yields
$$\mathscr{X}a_{\mathbf{m}} = (C(\mathbf{m}^*) - C(\mathbf{m})) b_0 a_\mathbf{m}.$$
From Lemma \ref{lem:a.3},  let $n = C(\mathbf{m}^*) - C(\mathbf{m})$, $a = {a_{\mathbf{m}}}^{|n|/n}$, and
$$B = (\delta_2 y)^{|n|} - a,$$
then $\{X, B\}$ is a nontrivial expansion of $\mathscr{X}$, and hence (2) is proved.

(3). If $r = 2$, let $\mathbf{m}^* = (m_0^*, m_1^*,m_2^*)$  and $\mathbf{m} =
\Delta_{1,2}(\mathbf{m}) = (m_0^*, m_1^*+1,m_2^* - 1)$. It is easy to verify $\mathcal{P}(\mathbf{m}) =
\{\mathbf{m}^*\}$. Thus, from Lemma \ref{lem:a.8}, we have
$$\mathscr{X}a_{\mathbf{m}} = b_0 a_{\mathbf{m}} - m_2^* c_{1,2}  b_1.$$
Here, we note $C(\mathbf{m}^*) - C(\mathbf{m}) = 1$ and $a_{\mathbf{m}^*} = 1$. Let
$$a = -\dfrac{a_{\mathbf{m}}}{m_2^*  c_{1,2}},$$
then $a$ satisfies
$$\mathscr{X} a = b_0 a + b_1.$$
From Lemma \ref{lem:a.4}, let
$$B = \delta_2^2 y - a \delta_2 y,$$
then the differential ideal $\{X, B\}$ is a nontrivial expansion of $\mathscr{X}$.

(4). If $r = 3$, we write $\mathbf{m} = (m_0^*, m_1^*, m_2^*, m_3^*)$. From Lemma
\ref{lem:a.11}, we have $m_1^*>0$ and $m_2^* = 0$, i.e., $\mathbf{m}^* = (m_0^*, m_1^*, 0, m_3^*)$. Let
$$\mathbf{p} = \Delta_{1,3}(\mathbf{m}^*) = (m_0^*, m_1^*, 1, m_3^*-1),$$
$$ \mathbf{m} = {\Delta_{1,2}^{-1}}(\mathbf{p}) = (m_0^*, m_1^*-1,2, m_3^*-1).$$ 
$$ \mathbf{q} = \Delta_{1,2}(\mathbf{p}) = (m_0^*, m_1^*+1,0,m_3^*-1),$$
It is easy to have
$C(\mathbf{m}) = C(\mathbf{m}^*)$. Therefore, from Lemma
\ref{lem:a.10}, $a_{\mathbf{m}}$ is a constant. Furthermore, we have 
$\mathcal{P}(\mathbf{p}) \subseteq \{\mathbf{m}^*,\mathbf{m}\}$ and
$\mathcal{P}(\mathbf{q}) \subseteq \{\mathbf{m}^*,\mathbf{p}\}$.  

Applying Lemma \ref{lem:a.8} to $a_{\mathbf{p}}$ and $a_{\mathbf{q}}$, respectively, and notice that $C(\mathbf{m}^*) - C(\mathbf{p}) =1 $ and $C(\mathbf{m}^*) - C(\mathbf{q}) = 2$, we have
\begin{equation}
\label{eq:a.6} 
\mathscr{X}a_{\mathbf{p}} = b_0 a_{\mathbf{p}} -
(m_3^* c_{1,3} a_{\mathbf{m}^*} + 2 c_{1, 2} a_{\mathbf{m}}) b_1,
\end{equation}
and
\begin{equation}
\label{eq:a.7} 
\mathscr{X}a_{\mathbf{q}} = 2 b_0 a_{\mathbf{q}} -
(m_3^*  c_{2,3}  b_2 a_{\mathbf{m}^*} + c_{1,2} b_1 a_{\mathbf{p}}).
\end{equation}
Since $a_{\mathbf{m}^*}$ and $a_\mathbf{m}$ are constants, we must have   
$m_3^* c_{1,3} a_{\mathbf{m}^*} + 2 c_{1,2}  a_{\mathbf{m}} = 0$
and $a_{\mathbf{p}} = 0$. If otherwise, we should have $r \leq 2$ from Lemma \ref{lem:a.3} or Lemma \ref{lem:a.4}.  

In \eqref{eq:a.7}, let $a_{\mathbf{p}} = 0$, $a_{\mathbf{m}^*}= 1$, and let
$$a = -\dfrac{a_{\mathbf{q}}}{m_3^* c_{2,3}},$$
then $a$ satisfies
$$\mathscr{X} a = 2 b_0 a + b_2.$$
From Lemma \ref{lem:a.15} and let
$$B = 2 (\delta_2y)(\delta_2^3 y) - 3 (\delta_2^2 y)^2 - a (\delta_2 y)^2,$$
the differential ideal $\{X, B\}$ is a nontrivial expansion of $\mathscr{X}$, and (4) is proved.

(5). If $r > 3$, we will show that $r = \infty$. If otherwise, $r$ is finite, then  Lemma \ref{lem:a.11} yields $m_1^*>0$ and $m_2^* = 0$, and therefore $\mathbf{m}^*$ can be written as 
$$\mathbf{m}^* = (m_0^*, m_1^*, 0,\cdots, 0, m_k^*, \cdots, m_r^*),$$
where $2 < k \leq r$ and $m_1^*, m_k^* >0 $.  We have the following.
\begin{enumerate}
\item[(a)] If $k = 3$, then 
$$\mathbf{m}^* = (m_0^*, m_1^*,0,m_3^*,m_4^*\cdots, m_r^*).$$
Let
\begin{eqnarray*}
\mathbf{p}&=& \Delta_{1,3}(\mathbf{m}^*) = (m_0^*,m_1^*,1,m_3^*-1,m_4^*,\cdots,m_r^*)\\
\mathbf{m}&=&\Delta_{1,2}^{-1}(\mathbf{p}) =  (m_0^*,m_1^*-1,2,m_3^*-1,m_4^*\cdots,m_r^*)\\
\mathbf{q} &=& \Delta_{1,2}(\mathbf{p}) = (m_0^*,m_1^*+1,0,m_3^*-1,m_4^*\cdots,m_r^*).
\end{eqnarray*}
Then $\#(\mathbf{m}) = 0$, $\mathcal{P}(\mathbf{p}) \subseteq \{\mathbf{m}^*, \mathbf{m}\}$ and $\mathcal{P}(\mathbf{q}) \subseteq \{\mathbf{m}^*, \mathbf{p}\}$.  Following the discussions as in (4), we have $\mathrm{ord}(\mathscr{X})\leq 3$, which contradicts with $r>3$.

\item[(b)] If $k>3$, let
\begin{eqnarray*}
\mathbf{p}&=&\Delta_{1,k}(\mathbf{m}^*) =(m_0^*,m_1^*,0,\cdots,1,m_k^*-1,\cdots,m_r^*)\\
\mathbf{m}&=&\Delta_{k-2,k-1}^{-1}(\mathbf{p}) = (m_0^*,m_1^*-1,0,\cdots,2,m_k^*-1,m_r^*).
\end{eqnarray*}
Then $\mathcal{P}(\mathbf{p})\subseteq \{\mathbf{m}^*,\mathbf{m}\}$. Therefore,
$$\mathscr{X}a_{\mathbf{p}} = b_0 a_{\mathbf{p}} - (m_k^* c_{1,k} b_1+ 2 c_{k-2,k-1} b_{k-2} a_{\mathbf{m}}).$$
Thus, we have $a_{\mathbf{m}}\not=0$, i.e., $\mathbf{m}\in \mathcal{I}_A$, otherwise we should have $\mathrm{ord}(\mathscr{X}) \leq 2$ as previous.  Furthermore, we have
$$C(\mathbf{m}) = C(\mathbf{m}^*) + k-3 > C(\mathbf{m}^*),$$
which is in contradiction to Lemma \ref{lem:c}.
\end{enumerate}
Thus, the above arguments conclude that $r$ must be $\infty$, and the Theorem has been proved.
\end{proof}

\section{Applications}
\label{sec:appl}

In this section,  we will apply the previous results to study the classification of polynomial differential equations \eqref{eq:3} and give some examples.   

First, from the proof of Lemmas \ref{lem:a.3} - \ref{lem:a.15}, the explicit
method to determine the class of a polynomial differential equation  \eqref{eq:3}  can be given as follows.
\begin{thm}
\label{th:2}Consider the polynomial differential equation \eqref{eq:3}, let
\begin{equation}
\label{eq:Bi}
b_i = -X_1\delta_2^{i+1}(\frac{X_2}{X_1}),\quad  (i = 0,1,2)
\end{equation}
and $r$ to be the order of the corresponding differential operator \eqref{eq:1}, then
\begin{enumerate}
\item[(1)] $r=0$ if, and only if, $K$ contains a first integral of \eqref{eq:3}; 
\item[(2)] $r=1$ if, and only if, $K$ contains no first integral of \eqref{eq:3}, and there exists $a\in K\backslash\{0\}$, and $n\in
\mathbb{Z}\backslash\{0\}$, such that
\begin{equation}
\label{eq:31} \mathscr{X} a = n b_0 a.
\end{equation}
In this case, \eqref{eq:3} has an integrating factor 
\begin{equation}
\eta = \dfrac{a^{1/n}}{X_1}.
\end{equation}
 \item[(3)] $r=2$ if, and only if, \eqref{eq:31} is not satisfied by any $a\in K\backslash \{0\}$ and $n\in \mathbb{N}$, and there exists
$a\in K$, such that 
\begin{equation} 
\label{eq:32} 
\mathscr{X} a = b_0 a + b_1.
\end{equation}
In this case, \eqref{eq:3} has an integrating factor of the form
\begin{equation}
\label{eq:3.5}
\eta = \dfrac{1}{X_1}\exp\left[\int_{(x_1^0,x_2^0)}^{(x_1,x_2)} \dfrac{a}{X_1}\left( X_1 d x_2 -(X_2 a + b_0) d x_1 \right)\right].
\end{equation}
 \item[(4)] $r=3$ if, and only if, \eqref{eq:32} is not satisfied by any $a\in K$, and there exists $a\in K$, such that
\begin{equation}
\label{eq:33} \mathscr{X}a = 2 b_0 a +  b_2. 
\end{equation}
In this case, \eqref{eq:3} has an integrating factor of the form
\begin{equation}
\eta = \dfrac{1}{X_1} \exp\left[\int_{(x_1^0,x_2^0)}^{(x_1,x_2)}(-\delta_2 \dfrac{X_2}{X_1} - \dfrac{X_2}{X_1} u) d x_1 + u d x_2\right],
\end{equation}
where $u$ is a solution of following partial differential equations
\begin{equation}
\left\{
\begin{array}{rcl}
\delta_1 u &=& - \delta_2^2 \dfrac{X_2}{X_1} - \dfrac{X_2}{X_1}a - (\delta_2 \dfrac{X_2}{X_1}) u - \dfrac{1}{2}(\dfrac{X_2}{X_1})u^2\\
\delta_2 u &=& a + \dfrac{1}{2} u^2.
\end{array}
\right.
\end{equation}
 \item[(5)] $r=\infty$ if, and only if, \eqref{eq:33} is not satisfied by any $a\in K$.
\end{enumerate}
\end{thm}

The proof is straightforward from pervious section, and is omitted here. 

We will give examples for each of the classes in Theorem \ref{th:2}.

It is easy to see that all equations
\begin{equation}
\frac{d x_1}{d t} = 1,\quad \dfrac{d x_2}{d t} = p(x_1),
\end{equation} 
with $p(x_1)$ a polynomial, have order $r=0$. The general homogenous linear equations\footnote{Here by general we mean most equations of this form.}
\begin{equation}
\frac{d x_1}{d t} = 1,\quad \dfrac{d x_2}{d t} = p(x_1) x_2,
\end{equation}
with $p(x_1)$ a rational function, have order $r = 1$, and the general non homogenous linear equations
\begin{equation}
\frac{d x_1}{d t} = 1,\quad \dfrac{d x_2}{d t} = p(x_1) x_2 + q(x_1),
\end{equation}
where $p(x_1)$ and $q(x_1)$ are rational functions, have order $r = 2$. 

In following, we will show that the general Riccati equation is an example of order $r = 3$.
\begin{prop}
\label{cor:3} The general Riccati
equations
\begin{equation}
\label{eq:riccati}
\dfrac{d x_1}{d t} = 1,\quad \frac{d x_2}{d t} = p_2(x_1) x_2^2 + p_1(x_1) x_2 + p_0(x_1),
\end{equation}
where $p_i(x), (i=0,1,2)$ are rational functions, have order $r = 3$.
\end{prop}
\begin{proof} We have known that the general Riccati equation \eqref{eq:riccati} does not have Liouvillian first integral (refer \cite{L:41} and\cite{Singer:92}), and hence the order $r $ is either $3$ or $\infty$ according to \cite{Singer:92}. 

From the equation \eqref{eq:riccati}, we have $X_1 = 1$ and $X_2 = p_2(x_1) x_2^2 + p_1(x_1) x_2 + p_0(x_1).$
Thus, we have $b_2 = 0$ from \eqref{eq:Bi}, and the equation \eqref{eq:33} has solution $a = 0$, therefore the order is $3$.
\end{proof}

Finally, we will show an example of differential equation with order $r=\infty$.

Consider the van der Pol equation
\begin{equation}
\label{vdp} \left\{\begin{array}{rcl} \dot{x}_1 &=& x_2 - \mu
(\dfrac{x_1^3}{3} - x_1),\\
\dot{x}_2 &=& -x_1 \end{array}\right.
 \ \ \ (\mu \not = 0).
\end{equation}
The van der Pol equation is well known for its existence of a limit cycle. 
Following Lemma was proved independently by Cheng et al.\cite{Cheng:95} and Odani\cite{Oda:95}, respectively. 

\begin{lem} (\cite{Cheng:95} and \cite{Oda:95})
\label{lem:vdp} The system of the van der Pol equation \eqref{vdp} has no algebraic solution curves. In particular, the limit cycle is not algebraic.
\end{lem}

\begin{prop}
\label{cor:2} 
The order of the van der Pol equation \eqref{vdp} is $r=\infty$.
\end{prop}
\begin{proof}
Let
$$X_1(x_1,x_2) = x_2 - \mu
(\dfrac{x_1^3}{3} - x_1),\quad \ X_2(x_1,x_2) = -x_1,$$ 
then the equation
\eqref{eq:33} for the van der Pol equation \eqref{vdp} reads
\begin{equation}
\label{eq:37} X_1^3  \mathscr{X}a + 2 x_1 X_1^2  a + 6 x_1 = 0.
\end{equation}
We only need to show that \eqref{eq:37} has no rational
function solution $a$. 

If on the contrary, \eqref{eq:37} has a rational function
solution $a = a_1/a_2$, where $a_1, a_2$ are relatively prime
polynomials, then $a_1$ and $a_2$ satisfy
$$X_1^3 (a_2 \mathscr{X} a_1 - a_1 \mathscr{X}a_2) + 2 x_1 X_1^2 a_1 a_2 + 6  x_1 a_2^2 = 0,$$
i.e.
$$a_2 (X_1^3 \mathscr{X}a_1 + 2 x_1 X_1^2 a_1 + 6 x_1 a_2) =  a_1 X_1^3 \mathscr{X}a_2.$$
Hence, there exist a polynomial $c(x_1,x_2)$, such that
\begin{eqnarray}
\label{eq:38} X_1^3  \mathscr{X} a_2 &=& c  a_2,\\
\label{eq:39} X_1^3  \mathscr{X} a_1 &=& (c - 2  x_1  X_1^2) a_1 - 6
x_1 a_2.
\end{eqnarray}

Let $a_2 = X_1^k p$, where $k$ is the maximum integer such that the polynomial $p$ does not contain $X_1$ as a factor. Substitute $a_2$ into
\eqref{eq:38}, we have
$$X_1^3 \mathscr{X} p = p (c -  k   X_1^2 \mathscr{X} X_1).$$
Thus, $p| (X_1^3 \mathscr{X}p)$, and therewith $p | \mathscr{X}p$ because $X_1$ is a prime polynomial and $p$ does not contain $X_1$ as a factor. Therefore, either $p$ is a constant or the planar curve defined by $p(x_1,x_2)=0$ is an algebraic invariant curve of the van der Pol equation
\eqref{vdp}. However, Lemma \ref{lem:vdp} has excluded the latter case. Therefore, $p$ must be a constant. 

We can let $p = 1$ without loss of generality, and therefore
\begin{equation}
\label{eq:40} a_2 = X_1^k,\quad c = k X_1^2\, \mathscr{X}X_1.
\end{equation}
Substitute \eqref{eq:40} into \eqref{eq:39}, we have
\begin{equation}
\label{eq:41} X_1^3 \mathscr{X}a_1 = (k \mathscr{X}X_1 - 2  x_1
) X_1^2 a_1 - 6  x_1 X_1^k,\quad (k\geq 0).
\end{equation}
Note that
$$(k \mathscr{X}X_1 - 2 x_1) = -k \mu (x_1^2 - 1) X_1 - (k + 2) x_1,$$
\eqref{eq:41} can be rewritten as 
\begin{equation}
\label{eq:kk}
X_1^3 \mathscr{X}a_1 = -k \mu (x_1^2 - 1) X_1^3 a_1 - (k + 2) x_1 X_1^2 a_1 - 6 x_1 X_1^k.
\end{equation}
From \eqref{eq:kk}, we claim that $k = 2$. If otherwise, we should have  $X_1|(k+2) x_1a_1$ if $k > 2$, or $X_1| 6 x_1$ if $k < 2$, which are not possible. 

Let $k = 2$, then equation \eqref{eq:41} becomes
\begin{equation}
X_1 \mathscr{X}a_1 = (2 \mathscr{X}X_1 - 2  x_1
)  a_1 - 6  x_1,
\end{equation}
which gives
\begin{equation}
\label{eq:42} \begin{array}{rl} 
&\left(x_2 - \mu (\dfrac{x_1^3}{3} - x_1)\right) \left((x_2 -
\mu (\dfrac{x_1^3}{3} - x_1)) \dfrac{\partial a_1}{\partial x_1}
-x_1  \dfrac{\partial a_1}{\partial x_2}\right)\\
 = &\left(-2 \mu  (x_1^2 - 1) (x_2 - \mu (\dfrac{x_1^3}{3} - x_1)) - 4 x_1\right) a_1 - 6  x_1.
\end{array}
\end{equation}
Let
\begin{equation}
\label{eq:a1}
a_1(x_1,x_2) = \sum_{i = 0}^m h_i(x_2)  x_1^i,
\end{equation}
where $h_i(x_2)$ are polynomials and $h_m(x_2) \not=0$. Substituting
\eqref{eq:a1} into \eqref{eq:42}, and comparing the coefficient
of $x_1^{m+5}$, we have
$$\frac{1}{9}  \mu^2\, m\, h_m(x_2) = \frac{2}{3} \mu^2 h_m(x_2),$$
which implies $m = 6$. Hence, we have $7$ coefficients $h_i(x_2), (i
= 0,\cdots , 6)$ to be determined, which are all polynomials of $x_2$. Next, comparing the coefficients of $x_1^{i}\ ( 0\leq
i\leq 10)$, we obtain following 11 differential-algebra equations for the coefficients: 
\begin{eqnarray*}
0 &=& x_2  ( -2  \mu    h_0(x_2) +
x_2   h_1(x_2) )\\
0 &=& 6 - 2  ( -2 + {\mu }^2)    h_0(x_2)
+ 2  x_2^2   h_2(x_2) - x_2   h_0'(x_2)\\
0 &=& 2  \mu  x_2  h_0(x_2) - ( -4 + {\mu }^2 )  h_1(x_2) +
2  \mu  x_2  h_2(x_2) +
  3  x_2^2   h_3(x_2)\\
&&{} - \mu  h_0'(x_2) -
  x_2 h_1'(x_2)\\
 0 &=& \frac{8  \mu^2}{3}   h_0(x_2) + \frac{4  \mu  x_2 }{3}
 h_1(x_2) + 4   h_2(x_2) +
  4  \mu  x_2    h_3(x_2) + 4  x_2^2   h_4(x_2)\\
&&{} - \mu    h_1'(x_2) -
  x_2   h_2'(x_2)\\
0 &=& 2  {\mu }^2   h_1(x_2) + \frac{2  \mu  x_2 }{3} 
h_2(x_2) + 4  h_3(x_2) +
  {\mu }^2   h_3(x_2) + 6  \mu  x_2    h_4(x_2)\\
  &&{} + 5  x_2^2   h_5(x_2) +
  \frac{\mu }{3}   h_0'(x_2) - \mu    h_2'(x_2) -
  x_2   h_3'(x_2)\\
0 &=& \frac{1}{3}(-2  {\mu }^2   h_0(x_2) + 4  {\mu }^2 
h_2(x_2) + 12  h_4(x_2) +
    6  {\mu }^2   h_4(x_2) + 24  \mu  x_2    h_5(x_2)\\
    &&{} + 18  x_2^2   h_6(x_2) +
    \mu    h_1'(x_2) - 3  \mu  h_3'(x_2) -
    3  x_2   h_4'(x_2))\\
0 &=& \frac{1}{9} (-5  {\mu }^2   h_1(x_2) + 6  {\mu }^2 
h_3(x_2) - 6  \mu  x_2    h_4(x_2) +
    36   h_5(x_2) + 27  {\mu }^2   h_5(x_2)\\
    &&{} + 90  \mu x_2    h_6(x_2) +
    3  \mu    h_2'(x_2) - 9  \mu    h_4'(x_2) -
    9  x_2   h_5'(x_2))\\
0 &=& -\frac{4 {\mu }^2}{9}   h_2(x_2) - \frac{4  \mu  x_2
}{3}  h_5(x_2) + 4   h_6(x_2) +
  4  {\mu }^2   h_6(x_2) + \frac{\mu }{3}   h_3'(x_2)\\
  &&{} - \mu    h_5'(x_2) -
  x_2   h_6'(x_2)\\
0 &=& -\frac{\mu}{3} ( \mu    h_3(x_2) + 2  \mu    h_5(x_2) +
6  x_2  h_6(x_2) -
         h_4'(x_2) + 3   h_6'(x_2) )\\
0 &=& -\frac{\mu}{9} ( 2  \mu    h_4(x_2) + 12  \mu  h_6(x_2)
-
3   h_5'(x_2) )\\
0 &=&- \frac{\mu}{9} ( \mu    h_5(x_2) - 3   h_6'(x_2))
\end{eqnarray*}
The above equations yield the following
\begin{equation}
\label{eq:hm}
x_2 (3 x_2 h_5'(x_2) - 2 \mu h_4'(x_2)) = 2 \mu^3.
\end{equation}
But \eqref{eq:hm} can not be satisfied because $h_4(x_2)$
and $h_5(x_2)$ are polynomials, and the left hand side contains a factor $x_2$, while the right hand side does not. Thus, we conclude that \eqref{eq:37} has no rational function solution, and hence the order of the
van der Pol equation is infinity from Theorem \ref{th:2}.
\end{proof}

\section*{Appendix}

\begin{lem}
\label{le:app}
Consider following partial differential equations
\begin{equation}
\label{eq:app.1}
\left\{
\begin{array}{rcl}
\dfrac{\partial u}{\partial x_1} &=& f(x_1,x_2,u)\\
\dfrac{\partial u}{\partial x_2} &=& g(x_1,x_2,u)
\end{array}
\right.
\end{equation}
Let
$$D_1 = \dfrac{\partial\ }{\partial x_1} + f(x_1,x_2,u)\dfrac{\partial\ }{\partial u},\quad D_2 = \dfrac{\partial\ }{\partial x_2} + g(x_1,x_2,u)\dfrac{\partial\ }{\partial u}.$$
If the functions $f$ and $g$ are analytic, and satisfy
\begin{equation}
\label{eq:DD}
D_2 f(x_1,x_2,u) \equiv D_1 g(x_1,x_2,u),
\end{equation}
in a neighborhood of $(0,0,0)$,  then the equation \eqref{eq:app.1} has a unique solution $u = u(x_1,x_2)$ that is analytic on a neighborhood of $(0,0)$ and $u(0,0) = 0$.
\end{lem}
\begin{proof}
Without loss of generality, we assume that $f$ and $g$ are analytic in 
$$\Omega = \{(x,y,u)\in \mathbb{C}^3 \Big| |x_1| + |x_2| + |u| \leq \rho\},$$
where $\rho$ is positive.  Then we can write $f(x_1,x_2,u)$ and $g(x_1,x_2,u)$ as power series
\begin{equation}
f(x_1,x_2,u) = \sum_{i,j,k} f_{i,j,k} x_1^i x_2^ju^k
\end{equation}
and
\begin{equation}
g(x_1,x_2,u)=\sum_{i,j,k} g_{i,j,k} x_1^i x_2^j u^k,
\end{equation}
respectively, and these series are convergent in $\Omega$.

Let
\begin{equation}
\label{eq:app.u}
u(x_1,x_2) = \sum_{i=0}^\infty \sum_{j=0}^\infty u_{i,j} x_1^i x_2^j,\quad (u_{0,0}= 0),\end{equation}
and substitute it into \eqref{eq:app.1}, we have the following equations
\begin{eqnarray}
\label{eq:app.2}
\sum_{i,j} i u_{i,j}x_1^{i-1}x_2^j &=& \sum_{i,j,k} f_{i,j,k} x_1^i x_2^j (\sum_{p,q} u_{p,q}x_1^px_2^q)^k\\
\label{eq:app.3}
\sum_{i,j} j u_{i,j} x_1^i x_2^{j-1} &=& \sum_{i,j,k} g_{i,j,k} x_1^i x_2^j (\sum_{p,q} u_{p,q} x_1^p x_2^q)^k.
\end{eqnarray}
First, from \eqref{eq:app.2} and comparing the coefficients of the same degrees of $x_1^m,\ (m \geq 1)$, we have
\begin{equation}
\label{eq:app.4}
u_{1,0} = f_{0,0,0}, 
\end{equation} 
and
\begin{equation}
\label{eq:app.5}
u_{m,0} = \dfrac{1}{m!} D_1^{m-1} f(x_1,x_2,u(x_1,x_2))|_{(x_1,x_2) = (0,0)}.
\end{equation}
Next, from \eqref{eq:app.3} and comparing the coefficients of the same degrees of $x_1^m x_2^n\ (n\geq 1)$, we have
\begin{equation}
\label{eq:app.6}
u_{0,1} = g_{0,0,0}, 
\end{equation}
and
\begin{equation}
\label{eq:app.7}
u_{m,n} = \dfrac{1}{m! n!} D_1^mD_2^{n-1} g(x_1,x_2,u(x_1,x_2))|_{(x_1,x_2) =(0,0)}.
\end{equation}

The right hand side of \eqref{eq:app.5} is a polynomial of $u_{i,0}$ with $i<m$. Thus, the coefficients $u_{m, 0}\ (m>0)$ are well defined by \eqref{eq:app.4} and \eqref{eq:app.5} step by step. Similarly, the right hand side of \eqref{eq:app.7} is a polynomial of the coefficients $u_{i,j}$ with $i < n$, $j\leq m$ an $i+j \leq m+ n - 1$. Thus, the coefficients of form $u_{m,n}\ (n\geq 1)$ can be determined by \eqref{eq:app.6}, \eqref{eq:app.7}, and the coefficients $u_{m,0}$ obtained previously. Thus, the coefficients in the power series \eqref{eq:app.u} are well defined and unique. Convergency of this power series can be proved by the Method of Majorants as follows. 

Let
$$M = \max_{(x,y,u)\in \Omega} \{|f(x_1,x_2,u)|, |g(x_1,x_2,u)|\}, $$
then 
\begin{equation}
F(x,y,u) = \dfrac{M}{1 -\dfrac{x_1+x_2 + u}{\rho}}
\end{equation}
is a majorant function of both $f(x_1,x_2,u)$ and $g(x_1,x_2,u)$.  Thus, following equation
\begin{equation}
\label{eq:app.m}
\left\{
\begin{array}{rcl}
\dfrac{\partial u}{\partial x_1} &=& F(x_1,x_2,u)\\
\dfrac{\partial u}{\partial x_2} &=& F(x_1,x_2,u)
\end{array}
\right.
\end{equation}
majorize the equation \eqref{eq:app.1}. It is easy to verify that, the equation \eqref{eq:app.m} has an analytic solution $u(x_1,x_2) = U(x_1+x_2)$, with $U(z)$ the analytic solution of 
\begin{equation}
\dfrac{d U}{d z} = \dfrac{M}{1 - \dfrac{z + U}{\rho}},\quad U(0) = 0.
\end{equation}
Thus, the convergency of \eqref{eq:app.u} is concluded by the Method of Majorants. Therefore, the function $u(x_1,x_2)$ given by \eqref{eq:app.u} is well defined in $\Omega$.

Finally, we need to show that the function $u(x_1,x_2)$ obtained above satisfies \eqref{eq:app.1}. 
We note that  when $m \geq 1$, \eqref{eq:DD} yields
$$D_1^m D_2^{n-1} g = D_1^{m-1} D_2^{n-1} D_1 g = D_1^{m-1} D_2^{n-1} D_2 f = D_1^{m-1} D_2^n f.$$
Thus, \eqref{eq:app.7} is equivalents to
\begin{equation}
\label{eq:app.8}
u_{m,n} = \dfrac{1}{m! n!} D_1^{m-1} D_2^{n} f(x_1,x_2,u(x_1,x_2))|_{(x_1,x_2) = (0,0)}.
\end{equation}
Therefore, from \eqref{eq:app.4}-\eqref{eq:app.7} and \eqref{eq:app.8}, the function $u(x_1,x_2)$ satisfies both equations in \ref{eq:app.1}. The Lemma has been proved.
\end{proof}

\bibliographystyle{elsarticle-num}      
\bibliography{classification}   

\end{document}